\documentclass{amsart}
\usepackage{amsmath}
\usepackage{amsfonts,amscd,amssymb,amsthm}
\usepackage{graphicx}
\usepackage{ifthen}
\usepackage{verbatim}
\usepackage{yhmath}
\newtheorem{theorem}[equation]{Theorem}
\newtheorem{lemma}[equation]{Lemma}
\newtheorem{proposition}[equation]{Proposition}

\theoremstyle{remark}

\newtheorem{nonsec}[equation]{}

\numberwithin{equation}{section}

\newcommand{\R}{\mathbb{R}}

\newcommand{\Bn}{ {\mathbb{B}^n} }
\newcommand{\Rn}{ {\mathbb{R}^n} }

\newcommand{\beq}{\begin{equation}}
\newcommand{\eeq}{\end{equation}}

\newcommand{\K}{\mathcal{K}\,}

\begin{document}


\title[On isometries of conformally invariant metrics]
      {\bf  On isometries of conformally invariant metrics}

\author{Riku Kl\'en}
\address[Riku Kl\'en]{Department of Mathematics, University of Turku, 20014 Turku, Finland}
\email{riku.klen@utu.fi}

\author{Matti Vuorinen}
\address[Matti Vuorinen]{Department of Mathematics, University of Turku, 20014 Turku, Finland}
\email{vuorinen@utu.fi}

\author{Xiaohui Zhang}
\address[Xiaohui Zhang]{Department of Physics and  Mathematics, University of Eastern Finland, 80101 Joensuu, Finland}
\curraddr{}
\email{xiaohui.zhang@uef.fi}
\thanks{The third author's research was supported by the Academy of Finland project 268009.}

\date{}

\dedicatory{In Memoriam: J. Ferrand, 1918--2014}

\begin{abstract}
We prove that isometries in a conformally invariant metric of a general domain are quasiconformal.
In the particular case of the punctured space, we prove that isometries in this metric are M\"obius,
thus resolving a conjecture of Ferrand, Martin and Vuorinen \cite[p.\,200]{fmv} in this particular case.
\end{abstract}



\newcounter{minutes}\setcounter{minutes}{\time}
\divide\time by 60
\newcounter{hours}\setcounter{hours}{\time}
\multiply\time by 60 \addtocounter{minutes}{-\time}
\def\thefootnote{}
\footnotetext{
\texttt{\tiny File:~\jobname .tex,
          printed: \number\year-\number\month-\number\day,
          \thehours.\ifnum\theminutes<10{0}\fi\theminutes}
}
\makeatletter\def\thefootnote{\@arabic\c@footnote}\makeatother


\maketitle


\section{Introduction}

The conformal invariant $\lambda$ was introduced in \cite{f1} in the general  framework of non-compact Riemannian manifolds. In the elementary case of a domain $D$ in $\overline\Rn$, $n\geq2$, we can define the function $\lambda_D(x,y)$ for $x,y\in D$, $x\neq y$, as the infimum of moduli of some curve families
$$
  \lambda_D(x,y)=\inf_{C_x,C_y}M(\Delta(C_x,C_y;D)),
$$
where $C_z$ is any curve in $D$ joining a point $z\in D$ and the boundary of $D$.
For the unit ball $\Bn$, the value of $\lambda_{\Bn}(x,y)$  can be explicitly expressed in terms of the hyperbolic distance $\rho(x,y)$ between $x$ and $y$ and the capacity  $\tau_n$ of the Teichm\"uller condenser as follows \cite[Thm 8.6]{vu1}
\begin{equation}\label{lambdaball}
\lambda_{\Bn}(x,y)=\frac{1}{2}\tau_n\left(\sinh^2\frac{\rho(x,y)}{2}\right).
\end{equation}
It is easy to see that if $H$ is a subdomain of $D$ and $x,y\in H$, we have
$$
  \lambda_H(x,y)\leq\lambda_D(x,y).
$$

As shown in \cite{f2}, for an arbitrary subdomain $D$ of $\Rn$ with $\mbox{Card}(\Rn\setminus D)\geq2$,
 $\lambda_D(x,y)^{1/(1-n)}$ is always a metric.
We say that a mapping $f:D\to f(D)\subset H$ is $\lambda$-Lipschitz if there exists a constant $K$ such that
$$
  \lambda_D(x,y)\leq K\lambda_H(f(x),f(y))
$$
for all $x,y\in D$. Because $(D,\lambda_D^{1/(1-n)})$ is a metric space, this terminology makes sense.
Ferrand, Martin and Vuorinen \cite{fmv} showed that $\lambda$-Lipschitz mappings contain quasiconformal mappings as a proper subclass. They proved, however, that many important properties of quasiconformal mappings, such as H\"older continuity in the Euclidean metric and the Schwarz-type lemma, also hold for these more general mappings. By the Riemann mapping theorem, it is easy to see from the formula \eqref{lambdaball} that a $\lambda$-isometry is conformal for $n=2$ when the domain is simply connected.
For the case $n\geq3,$ isometries
$$f : (\Bn, \lambda_{\Bn}^{1/(1-n)} ) \to (\Bn, \lambda_{\Bn}^{1/(1-n)} )$$
are also conformal, in fact M\"obius transformations.
 It was conjectured by Ferrand, Martin, and Vuorinen \cite{fmv} that a $\lambda$-isometry is always conformal for general proper subdomains of $\Rn$.

Unable to settle this conjecture, we will prove that $\lambda$-isometries are at least quasiconformal. We will also prove the conjecture in the special case of punctured space. Our results are as follows.

\begin{theorem}\label{lambdaiso}
Let $D$ and $D'$ be two proper subdomains of $\Rn$, and  let $f:D\to f(D)=D'$ be a homeomorphism which is an isometry
of $(D,\lambda_D)$ onto $(D,\lambda_{D'})$. Then $f$ is quasiconformal with linear dilatation $H(f)\leq256$.
\end{theorem}

\begin{theorem}\label{thm-punctmob}
Let $f:\R^n\setminus\{z_0\}\to\R^n$ be an isometry of $(\R^n\setminus\{z_0\},\lambda_{\R^n\setminus\{z_0\}})$ onto $(f(\R^n\setminus\{z_0\}),\lambda_{f(\R^n\setminus\{z_0\})})$. Then $f$ is a M\"obius transformation.
\end{theorem}

\section{Notation and preliminaries}

We use the notation $\Bn(x,r)=\{z\in\Rn: |x-z|<r \}$ for the Euclidean ball, and its boundary is the sphere $S^{n-1}(x,r)=\partial \Bn(x,r)$. We abbreviate $\Bn(r)=\Bn(0,r)$, $\Bn=\Bn(1)$, $S^{n-1}(r)=S^{n-1}(0,r)$ and $S^{n-1}=S^{n-1}(1)$.

We next recall some basic facts from \cite{vu1}.
Let $\Gamma$ be a curve family in $\overline\Rn$, $n\geq2\,,$ and denote its modulus by $M(\Gamma)$.  If $E,F,D\subset\overline\Rn$, then $\Delta(E,F;D)$ denotes the family of all curves joining $E$ to $F$ in $D$.
Set $\Delta(E,F)=\Delta(E,F;D)$ if $D=\Rn$ or $\overline\Rn$. The capacities of the Gr\"otzsch and Teichm\"uller condensers $\gamma_n(t), t>1$, and $\tau_n(s), s>0$, respectively, are defined by
$$
  \gamma_n(t)=M(\Delta(\overline\Bn,[te_1,\infty])),
$$
and
$$
  \tau_n(s)=M(\Delta([-e_1,0],[se_1,\infty])),
$$
where $e_1=(1,0,\cdots,0)\in\Rn$.
These strictly decreasing functions satisfy the basic identity \cite[5.53]{vu1}
\beq\label{gammatau}
  \gamma_n(t)=2^{n-1}\tau_n(t^2-1), \,\, t>1.
\eeq

Next we define some related distortion functions. The function $\varphi_{K,n}:[0,1]\to[0,1]$, $K>0$,
is the homeomorphism defined by $\varphi_{K,n}(0)=0, \varphi_{K,n}(1)=1$ and
\beq\label{phi}
  \varphi_{K,n}(r)=\frac{1}{\gamma_n^{-1}(K\gamma_n(1/r))}, \quad r\in(0,1).
\eeq
Let
$$
  \psi_{K,n}(r)=\sqrt{1-\varphi_{1/K,n}(\sqrt{1-r^2}\,)^2}.
$$
It follows easily from (\ref{gammatau}) and (\ref{phi}) that for $t>0$,
\beq\label{tauphi}
  \tau_n^{-1}(K\tau_n(t))=\frac{1-\varphi_{K,n}(1/\sqrt{1+t}\,)^2}{\varphi_{K,n}(1/\sqrt{1+t}\,)^2}.
\eeq

For each $x\in\Rn\setminus\{0,e_1\}$, we denote
$$p(x)=\inf M(\Delta(E,F)),$$
where the infimum is taken over all pairs of continua $E$ and $F$ with $0, e_1\in E$ and $x, \infty\in F$. It follows from two spherical symmetrizations with centers $0$ and $e_1$, respectively, that
\begin{equation}\label{pxlower}
p(x)\geq\tau_n(\min\{|x|,|x-e_1|\}).
\end{equation}
For the punctured space $\Rn\setminus\{0\}$, it is easy to check that for $x,y\in\Rn\setminus\{0\}$,
\begin{equation}\label{lambdainp}
\lambda_{\Rn\setminus\{0\}}(x,y)=\min\{p(r_x(y)),p(r_y(x))\},
\end{equation}
where $r_z$ is a similarity map with $r_z(z)=e_1$ and $r_z(0)=0$.
The reader is referred to \cite[Chapter 15]{avv97}  and \cite{hv} for more details of the function $p$.

\section{Proof of main theorems}

Now we turn to the proof of main results.

\begin{lemma}\label{lambda-estimate}
Let $G$ be a proper subdomain  of $\Rn$, let $x\in G$, $d_G(x)=d(x,\partial G)$, $B_x=\Bn(x,d_G(x))$, let $y\in B_x$ with $y\neq x$, and let $r_G(x,y)=|x-y|/d_G(x)$. Then we have
$$\frac12\tau_n\left(\frac{r_G(x,y)^2}{1-r_G(x,y)^2}\right)\leq\lambda_{G}(x,y)\leq\tau_n\left(\frac{r_G(x,y)}2\right).$$
\end{lemma}

\begin{proof}
The left-hand side of the inequality is well-known, see \cite[8.85 (2)(a)]{avv97}. Next, we prove the right-hand side of the inequality. Let $z\in\partial G$ with $d_G(x)=|x-z|$. By the domain monotonicity,  \cite[15.11]{avv97}, and conformal invariance of $\lambda_G$ we have
\allowdisplaybreaks\begin{align*}
\lambda_G(x,y)&\leq \lambda_{\Rn\setminus\{z\}}(x,y)=\lambda_{\Rn\setminus\{0\}}(x-z,y-z)\\
              &\leq \tau_n\left(\frac{|x-y|}{2\min\{|x-z|,|y-z|\}}\right)\\
              &\leq \tau_n\left(\frac{r_G(x,y)}2\right).\qedhere
\end{align*}
\end{proof}

For convenient reference we record the following lemma  from \cite[13.27]{avv97}.

\begin{lemma}\label{psiestimate}
For $n\geq2$, $0<r\leq1$ and $K\geq1$, we have
$$2^{1-2K}\leq\frac{\psi_{1/K,n}(r)}{r^K}\leq1,$$
and
$$1\leq\frac{\psi_{K,n}(r)}{r^{1/K}}\leq2^{2-1/K}.$$
\end{lemma}

\begin{nonsec}{\bf Proof of Theorem \ref{lambdaiso}.}
{\rm
For given $x\in D$, let $0<r<d_D(x)$ be sufficiently small such that
$|f(x)-f(y)|<d_{D'}(f(x))$ whenever $|x-y|=r$.
Then, by Lemma \ref{lambda-estimate}, we have that
\allowdisplaybreaks\begin{align*}
|f(x)-f(y)|& \leq 2\tau_n^{-1}(\lambda_{D'}(f(x),f(y)))d_{D'}(f(x))\\
           &= 2\tau_n^{-1}(\lambda_{D}(x,y))d_{D'}(f(x))\\
           &\leq 2\tau_n^{-1}\left(\frac12\tau_n\left(\frac{r_D(x,y)^2}{1-r_D(x,y)^2}\right)\right)d_{D'}(f(x)),
\end{align*}
and similarly,
\allowdisplaybreaks\begin{align*}
|f(x)-f(y)|\geq{}& \sqrt{\frac{\tau_n^{-1}(2\lambda_{D'}(f(x),f(y)))}{1+\tau_n^{-1}(2\lambda_{D'}(f(x),f(y)))}}d_{D'}(f(x))\\
           ={}& \sqrt{\frac{\tau_n^{-1}(2\lambda_{D}(x,y))}{1+\tau_n^{-1}(2\lambda_{D}(x,y))}}d_{D'}(f(x))\\
           \geq{}& \sqrt{\frac{\tau_n^{-1}(2\tau_n(r_D(x,y)/2))}{1+\tau_n^{-1}(2\tau_n(r_D(x,y)/2))}}d_{D'}(f(x)).
\end{align*}
Using these inequalities we  estimate the linear dilation of the $\lambda$-isometry.
\allowdisplaybreaks\begin{align*}
H(x,f)=&\limsup_{|x-y|=|x-z|=r\atop r\to0+}\frac{|f(x)-f(y)|}{|f(x)-f(z)|}\\
      \leq&\limsup_{|x-y|=|x-z|=r\atop r\to0+}{2\tau_n^{-1}\left(\frac12\tau_n\left(\frac{r_D(x,y)^2}{1-r_D(x,y)^2}\right)\right)}\left/
         {\sqrt{\frac{\tau_n^{-1}(2\tau_n(r_D(x,z)/2))}{1+\tau_n^{-1}(2\tau_n(r_D(x,z)/2))}}}\right.\\
      =&\limsup_{|x-y|=|x-z|=r\atop r\to0+}{2\tau_n^{-1}\left(\frac12\tau_n\left(\frac{r_D(x,y)^2}{1-r_D(x,y)^2}\right)\right)}\left/
         {\sqrt{\tau_n^{-1}(2\tau_n(r_D(x,z)/2))}}\right.\\
      =&\limsup_{|x-y|=|x-z|=r\atop r\to0+}\frac{2\left(1-\varphi_{1/2,n}(\sqrt{1-r_D(x,y)^2})^2\right)/\varphi_{1/2,n}(\sqrt{1-r_D(x,y)^2})^2}
         {\sqrt{\left(1-\varphi_{2,n}(\sqrt{2/(r_D(x,z)+2)})^2\right)\left/\varphi_{2,n}(\sqrt{2/(r_D(x,z)+2)})^2\right.}}\\
      =&\limsup_{|x-y|=|x-z|=r\atop r\to0+}\frac{2\left(1-\varphi_{1/2,n}(\sqrt{1-r_D(x,y)^2})^2\right)}{\sqrt{1-\varphi_{2,n}(\sqrt{2/(r_D(x,z)+2)})^2}}\\
      =&\limsup_{|x-y|=|x-z|=r\atop r\to0+}\frac{2\psi_{2,n}(r_D(x,y))^2}{\psi_{1/2,n}(\sqrt{r_D(x,z)/(r_D(x,z)+2)})}\\
      =&\limsup_{|x-y|=|x-z|=r\atop r\to0+}2\left(\frac{\psi_{2,n}(r_D(x,y))}{\sqrt{r_D(x,y)}}\right)^2
         \left(\frac{\psi_{1/2,n}(\sqrt{r_D(x,z)/(r_D(x,z)+2)})}{r_D(x,z)/(r_D(x,z)+2)}\right)^{-1}\\
       &\hspace{6cm}\cdot\frac{r_D(x,y)}{r_D(x,z)/(r_D(x,z)+2)}\\
      \leq{}&2(2^{2-1/2})^2(2^{1-2\cdot2})^{-1}2=256,
\end{align*}
where the last inequality follows from Lemma \ref{psiestimate}.
This completes the proof.} \hfill$\Box$
\end{nonsec}

For the special case of the punctured space $\Rn\setminus\{z_0\}$, we next give a better estimate for the linear dilatation.

\begin{lemma}\label{lem-punctqc}
Let $f:\Rn\setminus\{z_0\}\to\Rn$ be a homeomorphism which is an isometry of $(\Rn\setminus\{z_0\},\lambda_{\Rn\setminus\{z_0\}})$ onto
$(f(\Rn\setminus\{z_0\}),\lambda_{f(\Rn\setminus\{z_0\})})$. Then $f$ is quasiconformal with dilatation $H(f)\leq4$.
\end{lemma}

\begin{proof}
Since a $\lambda$-isometry is quasiconformal by Theorem \ref{lambdaiso}, the point $z_0$ is the isolated and removable singularity of
$f$, and we can extend $f$ to $\Rn$. It follows that $f(\Rn\setminus\{z_0\})$ is also a punctured space because $\Rn$ cannot be mapped quasiconformally onto a proper subdomain. By pre- and post-composition with M\"obius maps we may assume that $z_0=0$ and $f(0)=0$ since $\lambda$ is conformally invariant. Hence $f$ is a $\lambda$-isometry from
$\Rn\setminus\{0\}$ onto $\Rn\setminus\{0\}$. It is well known that \cite[15.13]{avv97}, for $x,y\in\Rn\setminus\{0\}$ and $x\neq y$,
\begin{equation}\label{lambda-estimatepunc}
  \tau_n\left(\frac{|x-y|}{m(x,y)}\right)\leq\lambda_{\Rn\setminus\{0\}}(x,y)\leq\tau_n\left(\frac{|x-y|}{2m(x,y)}\right), \quad m(x,y)=\min\{|x|,|y|\}.
\end{equation}
Then we have
$$
  \tau^{-1}_n(\lambda_{\Rn\setminus\{0\}}(x,y))m(x,y)\leq|x-y|\leq2\tau^{-1}_n(\lambda_{\Rn\setminus\{0\}}(x,y))m(x,y),
$$
and
\allowdisplaybreaks\begin{align*}
\frac{|f(x)-f(y)|}{|f(x)-f(z)|}&\leq \frac{2\tau^{-1}_n(\lambda_{\Rn\setminus\{0\}}(f(x),f(y)))m(f(x),f(y))}
                                       {\tau^{-1}_n(\lambda_{\Rn\setminus\{0\}}(f(x),f(z)))m(f(x),f(z))}\\
                               &= \frac{2\tau^{-1}_n(\lambda_{\Rn\setminus\{0\}}(x,y))m(f(x),f(y))}
                                       {\tau^{-1}_n(\lambda_{\Rn\setminus\{0\}}(x,z))m(f(x),f(z))}\\
                               &\leq \frac{2m(f(x),f(y))}{m(f(x),f(z))}\frac{|x-y|/m(x,y)}{|x-z|/(2m(x,z))}\\
                               &= \frac{4m(f(x),f(y))m(x,z)|x-y|}{m(f(x),f(z))m(x,y)|x-z|}.
\end{align*}
Hence
\[
  H(x,f)=\limsup_{|x-y|=|x-z|=r\atop r\to0}\frac{|f(x)-f(y)|}{|f(x)-f(z)|}\leq4.\qedhere
\]

\end{proof}

\begin{proposition}\label{prop-ball2ball}
Suppose that $f:\Rn\to\Rn$ is a homeomorphism with $f(0)=0$ and $f$ is a $\lambda_{\Rn\setminus\{0\}}$-isometry. Then $f$ maps each ball centered at the origin onto a ball centered at the origin.
\end{proposition}

\begin{proof}
The idea of this proof comes from \cite{tv}.  By precomposition with a stretching and a rotation we may assume that $f(e_1)=e_1$ and $\Bn\subset f(\Bn)$. Let $\mathcal{F}$ be the family of all $\lambda_{\Rn\setminus\{0\}}$-isometries $f:\Rn\to\Rn$ with the origin and $e_1$ fixed and  $\Bn\subset f(\Bn)$. Then $\mathcal{F}$ is a normal family since by Lemma \ref{lem-punctqc} all $f$ are 4-quasiconformal with two common fixed points  (see \cite[Corollary 19.3 and Theorem 20.5]{va}). Hence there exists a mapping $g\in \mathcal{F}$ such that
$$m_n(g(\Bn))=\max\{m_n(f(\Bn)): f\in\mathcal{F}\},$$
where $m_n$ is the Lebesgue measure. If $m_n(g(\Bn))>m_n(\Bn)$, then $\Bn\subsetneq g(\Bn)$, and therefore
$$g(\Bn)\subsetneq g\circ g(\Bn)\qquad \mbox{and}\qquad m_n(g\circ g(\Bn))>m_n(g(\Bn)).$$
However, it is easy to see that $g\circ g\in\mathcal{F}$ which is a contradiction. Hence we have that $m_n(g(\Bn))=m_n(\Bn)$, which implies $f(\Bn)=\Bn$. This completes the proof.
\end{proof}

\begin{lemma}\label{lem-oneintpoint}
Let $S^{n-1}(r)$ be the sphere with center at the origin  and radius $r>1$. Then for all $z\in S^{n-1}(r)\setminus\{re_1\}$,  $$\lambda_{\R^n\setminus\{0\}}(e_1,z)<\lambda_{\R^n\setminus\{0\}}(e_1,re_1).$$
\end{lemma}

\begin{proof}
It follows from \eqref{pxlower} that for $r>1$,
$$p(\frac{1}{r}e_1)\geq\tau_n(1-1/r)>\tau_n(r-1)=p(re_1),$$
which together with \eqref{lambdainp} implies that
$$\lambda_{\R^n\setminus\{0\}}(e_1,re_1)=\min\{p(re_1),p(\frac{1}{r}e_1)\}=p(re_1).$$
On the other hand, we have that
$$\lambda_{\R^n\setminus\{0\}}(e_1,z)\leq p(z)\leq M(\Delta([0,e_1],[z,\infty]))<M(\Delta([0,e_1],[re_1,\infty]))=p(re_1),$$
where the strict inequality follows from the polarization, see \cite[Theorem 1.1]{d} and its statement about equality.
Combining the above inequalities, we get the desired inequality.
\end{proof}

\medskip

\begin{figure}[h]

\begin{minipage}[t]{0.45\linewidth}
\centering
\includegraphics[width=6.5cm]{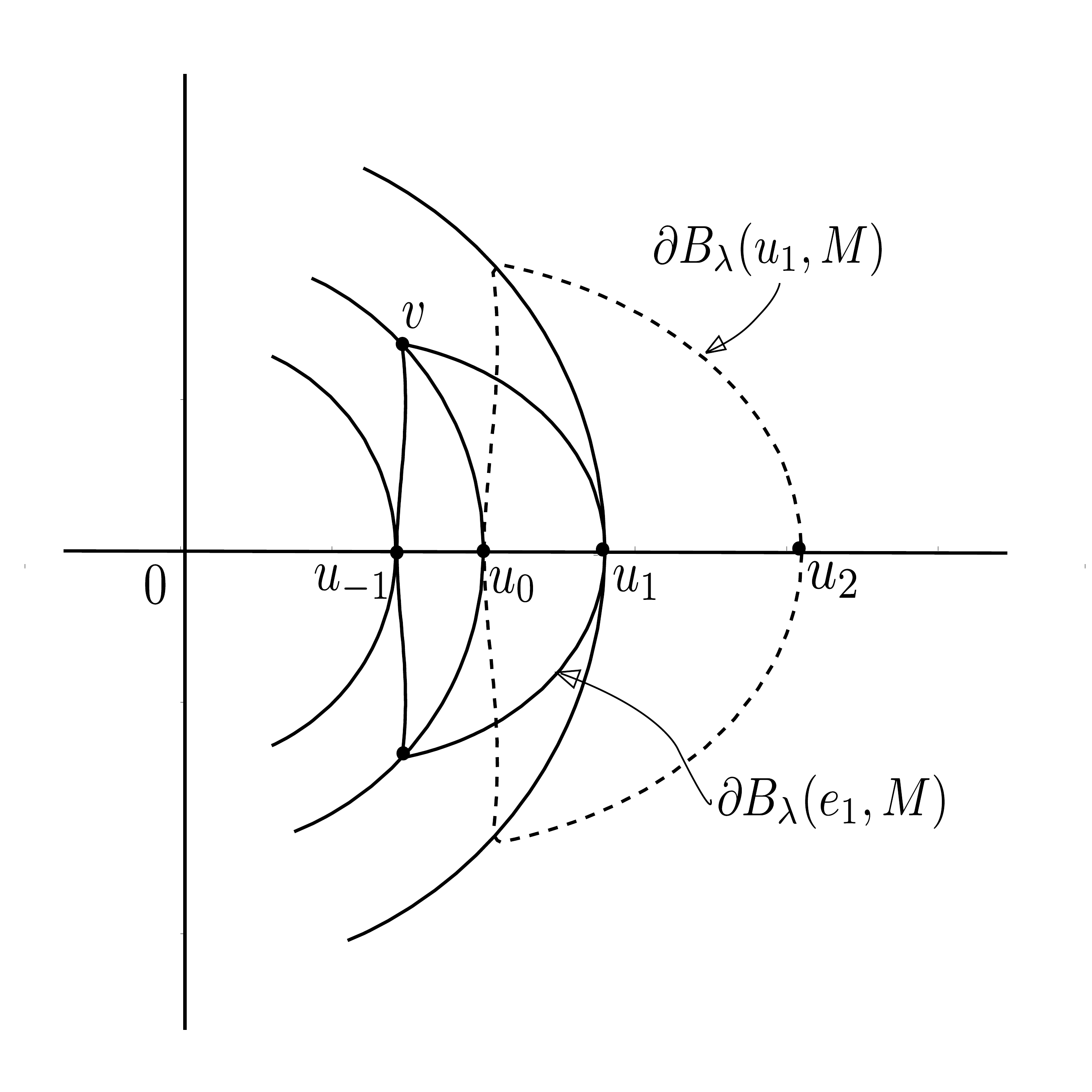}
\caption{The map $f$ is identical on $(0,\infty)$.}\label{figray}
\end{minipage}
\hspace{0.3cm}
\begin{minipage}[t]{0.45\linewidth}
\centering
\includegraphics[width=6.5cm]{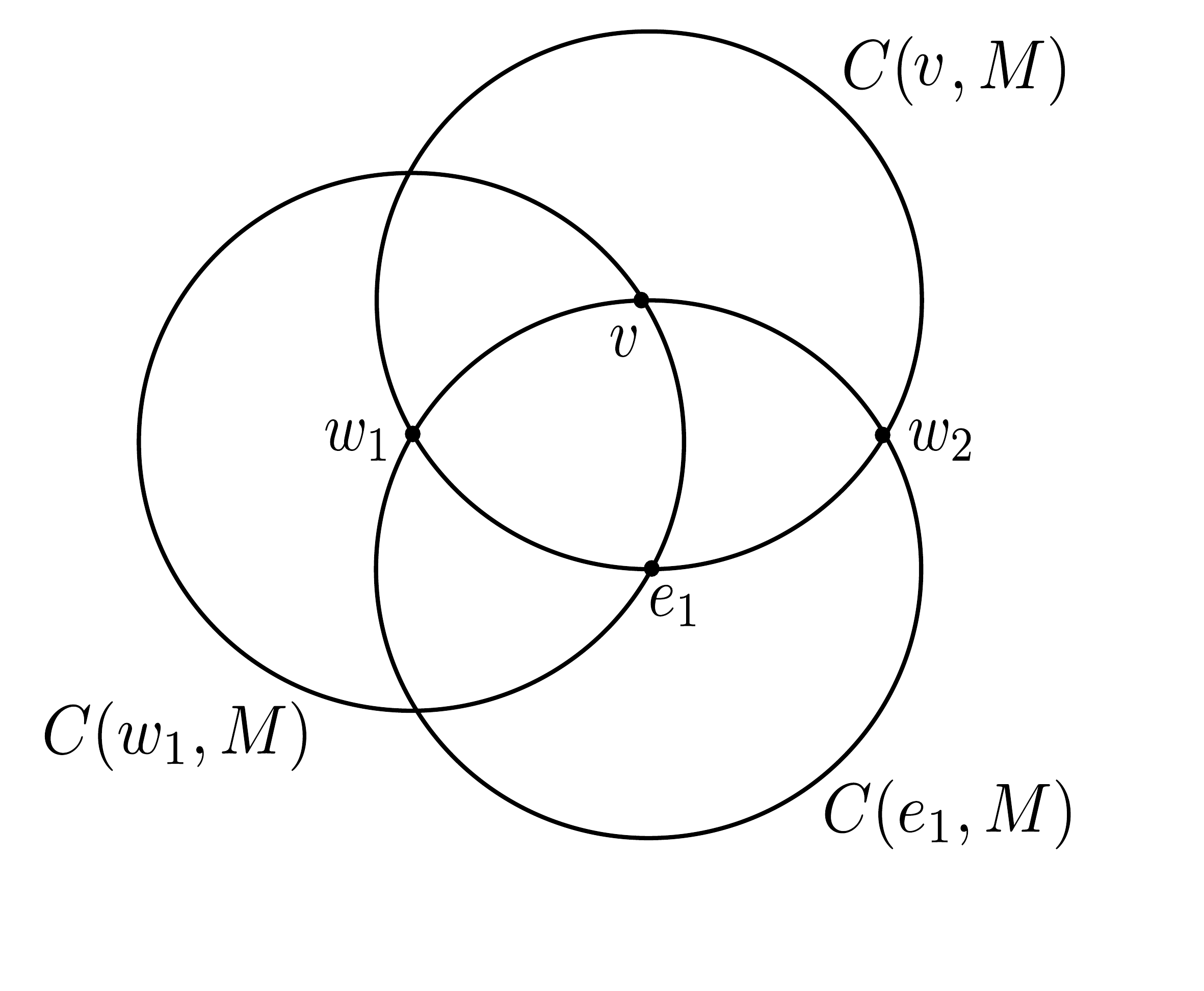}
\caption{The map $f$ is identical on the unit sphere.}\label{figsphere}
\end{minipage}

\end{figure}

\begin{nonsec}{\bf Proof of Theorem \ref{thm-punctmob}.}
{\rm
As in the proof of Lemma \ref{lem-punctqc}, by pre- and post-composition with M\"obius transformations we may assume that $f$ is an isometry of  $(\R^n\setminus\{0\},\lambda_{\R^n\setminus\{0\}})$ onto itself with the two points $0$ and $e_1$ fixed.

By the estimate \eqref{lambda-estimatepunc}, we choose a constant $M>0$ such that the metric ball $B_\lambda(e_1,M)=\{x\in\R^n: \lambda_{\R^n\setminus\{0\}}(e_1,x)>M\}$ is simply connected and $B_\lambda(e_1,M)\subset H_{+}=\{(x_1,\cdots,x_n):x_1>0\}$. Set $u_0=e_1$; for $k=0,1,2,\cdots$, set $u_{-(k+1)}=\partial B_\lambda(u_{-k},M)\cap(0,u_{-k})$ and  $u_{k+1}=\partial B_\lambda(u_{k},M)\cap(u_{k},\infty)$, see Figure \ref{figray}. It is easy to see that
$$\bigcup\limits_{k=-\infty}^\infty[u_k,u_{k+1}]=(0,\infty)=\{x\in\Rn:x=te_1,t>0\}.$$

By Lemma \ref{lem-oneintpoint}, we see that $\partial B_\lambda(e_1,M)\cap S^{n-1}(|u_1|)=\{u_1\}$. Additionally, since the ball  $B_\lambda(e_1,M)$ is symmetric with respect to the unit sphere,  we also have that $\partial B_\lambda(e_1,M)\cap S^{n-1}(|u_{-1}|)=\{u_{-1}\}$. Since the mapping $f$ is a $\lambda$-isometry and preserves the point $e_1$, it follows from Proposition \ref{prop-ball2ball} that $f$ maps the point $u_1$ possibly to itself or  to the point $u_{-1}$. By precomposition with the inversion with respect to the unit sphere, we may assume that $f(u_1)=u_1$. Hence, under this assumption we also get $f(u_{-1})=u_{-1}$.

We are going to prove by induction that $f$ is the identity map on the ray $(0,\infty)$. First, we prove that the restriction of $f$ on the interval $[u_{-1},u_1]$ is the identity map. The same argument as above applied to smaller metric balls $B_\lambda(e_1,K)$ with $K>M$ yields that for every $x\in(u_{-1},u_1)$, the image of $x$ under $f$ is possibly itself or its symmetric point with respect to the unit sphere. The latter case actually cannot happen since $f$ is a homeomorphism of $[u_{-1},u_1]$ with two endpoints fixed. Therefore, we see that  $f$ is the identity on the interval $[u_{-1},u_1]$ as desired. Next, suppose that the restriction of $f$ on the interval $[u_{k-1},u_{k+1}]$ is the identity map for some integer $k$. Since $\lambda_{\R^n\setminus\{0\}}$ is invariant under stretching, by considering the metric balls $B_{\lambda}(u_{k-1},K)$ and $B_{\lambda}(u_{k+1},K)$ for $K\geq M$ we see that the same argument as in the interval $[u_{-1},u_1]$ applied to the interval $[u_{k-2},u_{k}]$ and to the interval $[u_{k},u_{k+2}]$ yields that the restriction of $f$ on the intervals $[u_{k-2},u_{k}]$ and  $[u_{k},u_{k+2}]$ is the identity. Hence it follows from the induction that $f$ is the identity map on the whole ray $(0,\infty)$.

Next we prove that after composing with some M\"obius transformations, $f$ is also the identity map on the unit sphere $S^{n-1}$. Since the metric ball $B_\lambda(e_1,M)$ is symmetric with respect to the unit sphere and also symmetric under the rotations in the ray $(0,\infty)$, considering the simply connectedness of $B_\lambda(e_1,M)$ we see that $\partial B_\lambda(e_1,M)\cap S^{n-1}$ is a circle centered at $(0,\infty)$ which is in the plane perpendicular to $(0,\infty)$.  We denote this circle by $C(e_1,M)=\partial B_\lambda(e_1,M)\cap S^{n-1}$, see Figure \ref{figsphere}. Choose arbitrarily a point $v\in C(e_1,M)$.  We may assume that $f(v)=v$ by composition a rotation in the ray $(0,\infty)$ since $f(C(e_1,M))=C(e_1,M)$. Let $C(v,M)=\partial B_\lambda(v,M)\cap S^{n-1}$ and $C(v,M)\cap C(e_1,M)=\{w_1,w_2\}$. Then $f(\{w_1,w_2\})=\{w_1,w_2\}$ since we also have $f(C(v,M))=C(v,M)$. We may assume that $f(w_1)=w_1$ after composing the reflection with respect to the plane determined by the rays $(0,\infty)$ and $\{x\in\Rn: x=tv, t>0\}$. Under this assumption we also have $f(w_2)=w_2$. Similarly to the argument on the interval $[u_{-1},u_1]$, considering smaller balls $B_{\lambda}(e_1,K)$ for $K\geq M$ we see that $f$ is the identity map on the circular arc $\wideparen{w_1w_2}=C(v,M)\cap B_{\lambda}(e_1,M)$. The same argument yields that $f$ is the identity map on the circular arcs $C(v,K)\cap B_{\lambda}(e_1,M)$ for all $K>M$, and hence on the area in the unit sphere enclosed by the two circles $C(v,M)$ and $C(e_1,M)$. Repeat above argument by considering the metric balls centered at $e_1$ and $w_1$, $B_{\lambda}(e_1,K)$ and $B_{\lambda}(w_1,K)$, for all $K\geq M$. Note that we already have $f(v)=v$. We obtain that $f$ is also the identity map in the area in the unit sphere enclosed by the two circles $C(w_1,M)$ and $C(e_1,M)$. Applying the same argument by finite times, we see that $f$ is the identity map on the whole unit sphere.

Now, for an arbitrary point $x\in S^{n-1}$, $f(x)=x$. Note that we already have that $f$ is the identity map on $(0,\infty)$. The same argument as in the case of $(0,\infty)$ yields that $f$ is the identity map on the ray through the point $x$ emanating from the origin. Hence $f$ is the identity map of the punctured space $\R^n\setminus\{0\}$.

}
\end{nonsec}

\end{document}